\documentclass[12pt,letterpaper]{article}

\newcommand{\required}[1]{\section*{\hfil \sharp1\hfil}}

\usepackage{verbatim,url}
\usepackage{color}
\usepackage{amssymb,amsmath,amsfonts,amsthm}
\usepackage{hyperref}
\usepackage{upref}
\newcommand{\Beq}{\begin{equation}}
\newcommand{\Eeq}{\end{equation}}
\newcommand{\beq}{\begin{equation*}}
\newcommand{\eeq}{\end{equation*}}
\newcommand{\bal}{\begin{align}}
\newcommand{\eal}{\end{align}}

\newtheorem{theorem}{Theorem}
\newtheorem{Lemma}{Lemma}
\newtheorem{prop}{Proposition}
\newtheorem{definition}{Definition}

\theoremstyle{definition}
\newtheorem{claim}{Claim}

\newcommand{\ac}[1]{\begin{quotation}\textbf{Anuj's comment:\
		}{\textit{\sharp1}}\end{quotation}}
\newcommand{\rc}[1]{\begin{quotation}\textbf{Rohit's comment:\
		}{\textit{\sharp1}}\end{quotation}}

\title{ Support theorem for the transverse ray transform of tensor fields of rank 2}
\date{}
\author{Anuj Abhishek }

\begin{document}
	\maketitle
	\begin{abstract}
	\noindent Let ($M,g$) be a simple, real analytic, Riemannian manifold with boundary and of dimension $n\geq 3$. In this work, we prove a support theorem for the transverse ray transform of tensor fields of rank $2$ defined over such manifolds. More specifically, given a symmetric tensor field $f$ of rank 2, we show that if the transverse ray transform of $f$ vanishes over an appropriate open set of maximal geodesics of $M$, then the support of $f$ vanishes on the points of $M$ that lie on the union of the aforementioned open set of geodesics. 
	\end{abstract}
\section{Introduction}

In this article, our main goal is to prove a support theorem for the transverse ray transform (TRT) of a symmetric tensor field of rank 2. transverse ray transforms of such tensor fields appear quite naturally in the study of polarization tomography. The general physical principle behind polarization tomography is fairly simple to understand. The anisotropy in the medium characteristics, such as magnetic permeability tensor and dielectric permeability tensor, polarizes the electromagnetic waves passing through it. By measuring the polarization of a large number of rays passing through the medium, one is then able to detect and measure the anisotropy in the medium characteristics. Due to the transverse nature of electromagnetic rays, the polarization measurements along a ray which is obtained in the form of an integral along that ray depends only on the component of the desired medium characteristic that is ``transverse" to the ray direction. Hence the central problem of polarization tomography is to reconstruct a medium characteristic from the data which is in the form of the transverse ray transform of the medium characteristic. For a more detailed discussion, see \cite[Chapter 5]{Sharafutdinov_Book}.\\
 \par \noindent Consider a simple, real analytic, Riemannian manifold $M$ of dimension $n\geq 3$ with an analytic metric $g$. Let $[0,l(\gamma)]\ni t \mapsto \gamma(t)$ be a geodesic of the manifold $M$ with end points on the boundary and $\eta(t)$ be a vector field parallel along $\gamma(t)$ and orthogonal to $\dot{\gamma}(t)$ for every $t$. Sharafutdinov defines the transverse ray transform of a symmetric tensor field $f$ of rank 2 as \cite [Chapter 5]{Sharafutdinov_Book}:
$$ Jf(\gamma,\eta)=\int_{0}^{l(\gamma)}f_{ij}(\gamma(t))\eta^{i}(t)\eta^{j}(t)dt\quad \quad (\eta(t)\in \gamma^{\perp}(t))$$
 \noindent Such transforms have been studied by several authors, see e.g.\ \cite{Sharafutdinov_Book},\cite{Novikov2007},\cite {Desai2016},\cite{Lionheart2015a},\cite{Hammer2004}, \cite{Sharafutdinov2007}, in the context of a multitude of physical problems like polarization tomography, diffraction strain tomography and other such imaging modalities. In \cite{Sharafutdinov_Book}, Sharafutdinov proves an injectivity result for the transverse ray transform on a Compact Dissipative Riemannian manifold (CDRM). After the pioneering work of Sharafutdinov\cite{Sharafutdinov_Book} and Lionheart and Withers \cite{Lionheart2015a} who provided reconstruction methods for transverse ray transforms, recently, Desai and Lionheart have given an improved reconstruction algorithm for transverse ray transform for symmetric $2$- tensor fields in the Euclidean setting, see \cite{Desai2016a}. \newline
 \noindent In this work, we prove a support theorem for transverse ray transform of symmetric $2$- tensor fields defined on a compact, simple, real analytic Riemannian manifold. Apart from their theoretical significance, support theorems are useful for practical reasons in various tomography problems. Having a support theorem for an integral transform of a function or a tensor field tells us that we can reconstruct the desired function or the tensor field in the exterior of a given region solely by tomographic measurements in the exterior of the given region. The injectivity result for such transforms follows as a corollary of our more general result.  We use the tools of analytic microlocal analysis to prove our results. Such techniques have been extensively used to prove injectivity results and support theorems for very general Radon Transforms and X-ray transforms by several authors, among which we give below a partial list. Analytic microlocal techniques were first used by Boman and Quinto to prove a support theorem for Radon Transforms with real analytic weights in \cite{BQ}. Stefanov and Uhlmann have used analytic microlocal analysis to prove an s-injectivity result for geodesic ray transform of symmetric $2$-tensor fields in \cite{SU2}. Krishnan in \cite{K1} and Krishnan and Stefanov in \cite{KS} prove support theorems for geodesic ray transform of functions and symmetric $2$- tensor fields respectively. Expanding on these works, the authors in \cite{Abhishek2017} proved a support theorem for integral moments of symmetric $m$- tensor fields.\newline
 \noindent The organization of the current paper is as follows: In section $2$, we give some definitions and our main theorem in this work. In section $3$, we prove a preliminary result that shows the analytic dependence of a normal vector field on the geodesic along which it is translated in a parallel manner. In section $4$, we prove a microlocal proposition which is an analogue of \cite [Proposition 2]{SU3}. The proof of our main theorem is given in section $5$ along the lines of the proof given by Krishnan in \cite [section 3]{K1}.
 \paragraph{Acknowledgment:}
 I would like to thank Prof. Todd Quinto and Prof. Venky Krishnan for their help and encouragement during the writing of the paper. 

  \section{Definitions and the Main Theorem}
  
  \begin{definition} [Simple Manifold]
  	A compact Riemannian manifold (M,g) with smooth boundary is said to be simple if:\\
  	(i) The boundary of the manifold $\partial M$ is strictly convex: $\langle \nabla_\xi\nu(x),\xi\rangle>0$ for all $\xi \in T_x(\partial M)$ and where $\nu(x)$ is the unit outward pointing normal to the boundary.\\
  	(ii) The map $\exp_x:\exp^{-1}_x(M)\mapsto M$ is a diffeomorphism for all $x\in M$.
  \end{definition}
\noindent Here the second condition implies that any two points $x$ and $y$ in the manifold $M$ are connected by a unique geodesic which depends smoothly on the points $x$ and $y$. It is well-known that any such simple manifold is necessarily diffeomorphic to a ball in $\mathbb{R}^n$, see \cite{Sharafutdinov_Book}. Hence in the following analysis, we can assume that the manifold is some domain in $\mathbb{R}^n$. In this article, we work with a fixed simple Riemannian manifold $M$ with a fixed real analytic atlas and a given metric $g$ which is also assumed to be real analytic. Note that a tensor field is said to be real analytic on a set $U$ if it is real analytic in a neighbourhood of the set $U$. Furthermore we will work with symmetric tensor fields $f$ of order 2 on $M$ i.e. $f \in S^2(M)$. In coordinate representation, the tensor field $f=f_{ij}$. We will also assume Einstein convention of summing over any repeated indices and raise or lower index on a tensor field via the metric $g$. As such,$f_{ij}$ and $f^{ij}=f_{kl}g^{ik}g^{jl}$ will be thought of as equivalent representations of the same tensor field.
\par \noindent Let $\widetilde{M}$ be a real analytic extension of $M$ such that $g$ also extends analytically to $\widetilde{M}$. We extend the tensor fields $f\in S^2(M)$ by 0 in $\widetilde{M}\setminus M$. We will think of maximal geodesics in $M$ as restriction of geodesics in $\widetilde{M}$ with distinct end points in $\widetilde{M}\setminus M$. A geodesic will be denoted by $\gamma$. 
\par \noindent Let $\mathcal{A}$ be an open set of geodesics in $\widetilde{M}$ with end  points in $\widetilde{M}\setminus M$. We assume that any geodesic in this set $\mathcal{A}$ is homotopically deformable within the set $\mathcal{A}$ to a geodesic outside $M$. By this we mean that there exists a continuous map which takes a geodesic of $\widetilde{M}$, say $\gamma_0 \in \mathcal{A}$  to some geodesic $\gamma_1 \in \mathcal{A}$ which lies completely outside $M$. We will denote by $M_{\mathcal{A}}$ the set of points of $M$ that belong to $\mathcal{A}$, i.e.\ $M_{\mathcal{A}}=\underset{\gamma \in \mathcal{A}}{\cup}\gamma$ and similarly $\partial_{\mathcal{A}}M=M_{\mathcal{A}}\cap \partial M$. 
\noindent Finally, with a slight abuse of notation we will denote by $\mathcal{E}^{\prime}(\widetilde{M})$, the set of tensor fields whose components are compactly supported distributions in the $int(\widetilde{M})$.
\noindent Now we are ready to state our main theorem in this work:

	\begin{theorem} \label{Main_theorem}Let $(M,g)$ be a simple real analytic Riemannian manifold of dimension $n\geq 3$ and $\widetilde{M}$ be a real analytic extension of $M$. Let $\mathcal{A}$ be any open set of geodesics of $\widetilde{M}$ such that each geodesic $\gamma \in \mathcal{A}$ is continuously deformable within the set $\mathcal{A}$ to some geodesic outside $M$. Let $f\in \mathcal{E}^{\prime}(\widetilde{M})$ be a symmetric tensor field of order 2 supported in $M$. If $Jf(\gamma,\eta)=0$ for every $\gamma \in \mathcal{A}$ and for every $\eta \in \gamma^{\perp}$, then $f=0$ on $M_{\mathcal{A}}$. 
	\end{theorem}
\noindent \textbf{Remark: } For the case $n=2$, observe that the data available by taking the transverse ray transform is the same as the data from the geodesic ray transform up to a diffeomorphism. We also recall that the geodesic ray transform has a non-trivial kernel, see e.g. \cite{Sharafutdinov_Book}.   

\section{Preliminary Results}
  
  Let $f$ be a symmetric tensor field defined on $M$, i.e.\ $f_{ij}=f_{ji}$. Then for each geodesic $\gamma$ with end points on the boundary $\partial M$, and for each $\eta \in \gamma^{\perp}$ where $\gamma^{\perp}$ is the space of vector fields formed by parallel translation along $\gamma$ and orthogonal to $\dot{\gamma}$, the transverse ray transform is given by the bilinear form: 
 $$Jf(\gamma,\eta)=\int_{0}^{l(\gamma)}f_{ij}(\gamma(t))\eta^{i}(t)\eta^{j}(t)dt\quad \quad (\eta(t)\in \gamma^{\perp}(t))$$

 \begin{Lemma}\label{eta_analytic}
  	The field $\eta$ that is orthogonal to the geodesic $\gamma$ and is parallel along it depends on $\gamma$ analytically.
  \end{Lemma}
\begin{proof}
	Note that $\eta(t)$ is a parallel translate along $\gamma$\cite[Page 151]{Sharafutdinov_Book} and satisfies\cite[Theorem 4.13]{Lee2006}:
	$$\dot{\eta}^k(t)=-\eta^j(t)\dot{\gamma}(t)\Gamma^k_{ij}(\gamma(t));\quad \quad \eta(0)=\eta_{0}$$
	where $\Gamma^k_{ij}$ represents Christoffel symbols. Let us rewrite this equation as:
	\begin{equation}\label{translate_eqn}
	\dot{\eta}^k(t)=F(\eta,\gamma,\dot{\gamma});\quad \quad \eta(0)=\eta_0
	\end{equation}
 $F(\eta,\dot{\gamma},\gamma)$ is analytic in its arguments because we assume that metric $g$ is analytic and hence Christoffel symbols $\Gamma^k_{ij}$ are analytic as well. Let us rename $\dot{\gamma}(t)$ as $\zeta(t)$ and $\gamma(t)$ as $z(t)$. Then we recast equation \ref{translate_eqn} in to the following system:
	\begin{align*}
	\dot{\eta}^k(t)&=F(\eta,\zeta,z);&\eta(0)=\eta_0\\
	\dot{\zeta}^k(t)&=\Gamma^k_{ij}z^i(t)z^j(t);  & \zeta(0)=\zeta_{0}(=\dot{\gamma}(0))\\
	\dot{z}^k(t)&=\zeta^k(t); &z(0)=z_{0}(=\gamma(0))
	\end{align*}

	\noindent (In writing $\dot{\zeta}^k(t)=\Gamma^k_{ij}z^i(t)z^j(t)$, we have made use of the geodesic equation, since $z(t)=\gamma(t)$ and $\zeta(t)=\dot{\gamma}(t)$.) Together the three equations can be rewritten as a new system:
	\begin{equation}
	\dot{\tilde{\eta}}(t)=\tilde{F}(\tilde{\eta});\quad \quad \tilde{\eta}(0)=\tilde{\eta}_0
	\end{equation}
	where $\tilde{\eta}=(\eta,\zeta,z)$ and $\tilde{F}$ is the RHS of the system written above. Clearly $\tilde{F}$ is also analytic. Hence by \cite[Proposition 6.2]{Taylor2010}, $\tilde{\eta}(t)$ depends analytically on initial conditions $\tilde{\eta}_0$. In particular, $\eta(t)$ depends analytically on $\gamma(0)$ and $\dot{\gamma}(0)$. But $\gamma(0)$ and $\dot{\gamma}(0)$ are the parameters (starting point and initial direction respectively) which uniquely determine a geodesic on the manifold. This shows that $\eta(t)$ depends on geodesics $\gamma$ analytically.
\end{proof}
\noindent Before we move on further, we would like to show how to extend the definition of the transverse ray transform to distribution valued tensor fields.  
\subsection*{Extension of the definition of the transverse ray transform to distribution valued tensor fields}
\noindent Let $\gamma=\gamma_{x,\theta}:[\tau_{-}(x,\theta),0]\to M$ be a maximal geodesic of $M$ with initial conditions $\gamma(0)=x$ and $\dot{\gamma}(0)=\theta$. Recall that these maximal geodesics can be thought of as restriction of geodesics of $\widetilde{M}$ with end points in $\widetilde{M} \setminus M$. Let us denote by $I^{t,0}_\gamma$, the operator of parallel translation along the geodesic $\gamma$ i.e.\ $I^{t,0}_\gamma:T_{\gamma(t)}M \mapsto T_{\gamma(0)}M$.  We also recall that this operator is a linear isomorphism between the respective tangent spaces.  Since $\eta(t)$ is a parallel translate along $\gamma(t)$, hence there exists a unique $\eta_0 \in T_{\gamma(\tau_{-})}M$ such that $\eta(t)=I^{0,t}_\gamma(\eta_0)$. Let $$\Gamma_{-}:=\big{\{}(x,\xi)\in TM|\ x\in \partial M, |\xi|=1,\langle\xi,\nu(x)\rangle < 0\big{\}}$$ where $\nu(x)$ is the unit outer normal to $\partial M$ at $x$. Next, consider the space of symmetric tensor fields of rank $2$ on $T^*M$ which will be denoted by: $S^2(T^*M)$. We know that there is a canonical embedding that identifies tensor fields on $M$ with tensor fields on $T^*M$ which are independent of the second argument, see \cite[3.4.7]{Sharafutdinov_Book}. Under this identification, the field $f\in S^2(M)$ will be identified with the corresponding field in $S^2(T^*M)$ which we will again refer to as $f$. Further, if we take the restriction of the projection operator for the tangent bundle, $p:\Gamma_{-}\to M$, then this induces a smooth map between the space of tensor fields, $p^*:S^2M \to S^2(\Gamma_{-})$ \cite[Proposition 11.9]{Lee2003}. Now using the above mentioned identification of tensor fields in $S^2(M)$ with tensor fields in $S^2(T^*M)$ which are independent of the second argument, let us consider the space of pullback of such tensor fields $p^*(S^2(T^*M))$ and denote it by $S^2\Pi_M$. Following Sharafutdinov, we define the operator:
$$\widetilde{J}:C^{\infty}(S^2(T^*M))\mapsto C^{\infty}(S^2\Pi_M) $$ where $\widetilde{J}$ is given by the relation
$$\widetilde Jf(x,\theta)=\int_{\tau_{-}(x,\theta)}^{0}I^{t,0}_\gamma(P_{\dot{\gamma}(t)}f(\gamma(t))dt \quad \quad (x,\theta)\in \Gamma_{-}$$ Here,
$(P_{\dot{\gamma}(t)}f)_{ij} = \big(\delta^k_i-\frac{1}{|\dot{\gamma}(t)|^2}\dot{\gamma}(t)_i\dot{\gamma}(t)^k\big)\big(\delta^l_j-\frac{1}{|\dot{\gamma}(t)|^2}\dot{\gamma}(t)_j\dot{\gamma}(t)^l\big)f_{kl}=\big((Id-\frac{\dot\gamma \dot{\gamma}^t}{|\dot\gamma|^2})f(Id-\frac{\dot\gamma \dot{\gamma}^t}{|\dot\gamma|^2})\big)_{ij}$.  From \cite[equation 5.2.5]{Sharafutdinov_Book}, we have the following: \begin{equation}
\langle \widetilde Jf (x,\theta),\eta_0 \otimes \eta_0\rangle= Jf(\gamma,\eta)
\end{equation} 
Thus in order to make sense of the transverse ray transform for distribution valued tensor fields, all we need to do is to interpret $\widetilde Jf (x,\theta)$ by duality for compactly supported distribution valued tensor fields $f$. For this we will need an expression for the adjoint $(\widetilde{J})^*$. First let $f\in L^2(M)$ and take any  $\phi(x,\xi) \in C_c^{\infty}(S^2\Pi_M)$. We will consider the following inner product:
\begin{align*}
(\widetilde{J}f,\phi)_{\Gamma_-}&=\int_{\Gamma_{-}}\bar{\phi}(x,\theta)\int_{\tau_{-}(x,\theta)}^{0}I^{t,0}_\gamma(P_{\dot{\gamma}(t)}f(\gamma(t)))dtd\mu(x,\theta)\\
\end{align*}
Let us now define a function $\phi^{\sharp}(\gamma(t),\dot{\gamma}(t))$ such that it is constant along the geodesic and is equal to $\bar{\phi}(x,\theta)$ on $\Gamma_{-}$, i.e. $$\nabla_{\dot\gamma}\phi^{\sharp}(\gamma(t),\dot{\gamma}(t))=0, \quad \quad \phi^{\sharp}(\gamma(0),\dot{\gamma}(0))=\bar{\phi}(x,\theta).$$ This means that $\phi^{\sharp}(\gamma(t),\dot{\gamma}(t))$ is formed by parallel translation of $\bar\phi(x,\theta)$ along the geodesic $\gamma(t)$. Then the above can be rewritten as:
\begin{align*}
(\widetilde{J}f,\phi)_{\Gamma_{-}}
&=\int_{\Gamma_{-}}\int_{\tau_{-}(x,\theta)}^{0}I^{t,0}_\gamma(P_{\dot{\gamma}(t)}f(\gamma(t)))\bar\phi(x,\theta)dtd\mu(x,\theta)\\
&=\int_{\Gamma_{-}}\int_{\tau_{-}(x,\theta)}^{0}P_{\dot{\gamma}(t)}f(\gamma(t))I^{0,t}_\gamma\bar\phi(x,\theta)dtd\mu(x,\theta)\\
&=\int_{\Gamma_{-}}\int_{\tau_{-}(x,\theta)}^{0}P_{\dot{\gamma}(t)}\big(f(\gamma(t))\big)\phi^{\sharp}(\gamma(t),\dot{\gamma}(t))dtd\mu(x,\theta)\\
&=\int_{\Gamma_{-}}\int_{\tau_{-}(x,\theta)}^{0}f(\gamma(t))P_{\dot{\gamma}(t)}\big(\phi^{\sharp}(\gamma(t),\dot{\gamma}(t))\big)dtd\mu(x,\theta)\\
\end{align*}
Now we apply Santalo's formula \cite[Lemma 3.3.2]{Sharafutdinov} to the above and get:
\begin{align*}
(\widetilde{J}f,\phi)_{\Gamma_{-}}&=\int_{SM}f(x)P_{\xi}\big(\phi^{\sharp}\big)d\sigma\quad \text{ where } d\sigma \text{ is a measure on }SM\\
&=\int_M f(x)\int_{S_xM}P_{\xi}\big(\phi^{\sharp}\big)d\sigma_x(\xi) d\text{ Vol}(x)\\
&=(f,({\widetilde{J}})^*\phi)_{L^2(M)}
\end{align*}
where we have the adjoint of $\widetilde{J}$ given by \begin{align*}
({\widetilde{J}})^*\phi&=\int_{S_xM}P_{\xi}\big(\phi^{\sharp}\big)d\sigma_x(\xi)\\
&=\int_{S_xM}(Id-\frac{\xi \xi^t}{|\xi|^2})\phi^{\sharp}(Id-\frac{\xi \xi^t}{|\xi|^2})d\sigma_x(\xi)
\end{align*} 
Now for a compactly supported distribution valued tensor field $f$ and for  $\phi \in C_c^{\infty}(S^2\Pi_M)$  we define,
$$\langle\widetilde{J}f,\phi\rangle:=\langle f, (\widetilde{J})^*\phi\rangle.$$
\textbf{Remark:} To understand $ Jf(\gamma,\eta)=\langle \widetilde Jf (x,\theta),\eta_0 \otimes \eta_0\rangle $ when $f$ is a compactly supported distribution, we multiply $\eta_0\otimes\eta_0$ by a compactly supported function $\Psi(x,\theta)$ such that $\Psi\times(\eta_0\otimes\eta_0)$ is in $C_c^\infty (S^2\Pi_M)$. Then we follow the procedure described above to interpret the tranverse ray transform for such fields.

 \section{A microlocal proposition}
 The following proposition is an analogue of \cite[Proposition 2]{SU3}.
 \begin{prop}\label{SUanalogue}
 	Let $(x_0,\xi_0) \in T^*M\setminus0$ and let $\gamma_0$ be a fixed simple geodesic through $x_0$ normal to $\xi_0$. Let $Jf(\gamma,\eta)=0$ for some symmetric 2-tensor $f \in \mathcal{E}^{\prime}(\widetilde{M})$ supported in $M$ and for all $\gamma \in nbd.(\gamma_0)$ and for all $\eta \in \gamma^{\perp}$. Then $$(x_0,\xi_0)\notin WF_A(f).$$
 \end{prop}
\begin{proof}
	The following argument is adapted from the proof of \cite [ Proposition 2] {SU3}. For the compact, simple, real analytic Riemannian manifold $M$, we can construct analytical semi geodesic coordinates $(x^\prime,x^n)$ in some tubular neighbourhood $U$ of $\gamma_0$. This construction has been worked out and used by Stefanov and Uhlmann in many articles, see e.g.\ \cite[ section 2]{SU3}. Furthermore, we assume $x_0=0$ and $x^\prime=0$ on $\gamma_0$. We can represent $U$ in these co-ordinates by: $U=\{x:|x^\prime|\leq \epsilon\}$ for some $\epsilon<<1$. We can choose $\epsilon$ to be so that the $\{|x^{\prime}|\leq \epsilon; x^n=l^{\underset{-}{+}}\}$ is outside $M$. Then $\xi_0=((\xi_0)^\prime,0)$ with $(\xi_0)_n=0$. Our goal is to show:
	$$(0,\xi_0)\notin WF_A(f).$$ Consider $Z=\{|x|<\frac{7\epsilon}{8}: |x_{n}|=0\}$ and let $x^\prime$ variable be denoted on $Z$ by $z^\prime$. Then $(z^\prime,\theta^\prime)$ are local co-ordinates in nbd.$(\gamma_0)$ given by $(z^\prime,\theta^\prime)\to \gamma_{(z^\prime,0),(\theta^\prime,1)}$. Here, $|\theta^\prime|<<1$ (where, the geodesic is in the direction $(\theta^\prime,1)$).\\ \noindent Let $\{\chi_N(z{\prime})\}$ be a sequence of cutoff functions satisfying the estimates:$$\partial_{\alpha}(\chi_N)\leq (CN)^{|\alpha|};\quad \quad \text{for some }C \text{ and for }  |\alpha|<N.$$ Let $\theta=(\theta^{\prime},1)$. We multiply $$Jf(\gamma_{(z^{\prime},0;\theta )},\eta_{(z^{\prime},0;\theta )})=0$$ by $\chi_N(z^{\prime})e^{i\lambda z^{\prime}.\xi^{\prime}}$, where $\lambda>0$, $\xi^{\prime}$ is in a complex neighbourhood of $(\xi_0)^{\prime}$ and integrate it with respect to $z^{\prime}$ to get:$$\int e^{i\lambda z^\prime(x,\theta^\prime).\xi^\prime}\chi_{N}(x,\theta^\prime)f_{ij}(\gamma_{(z^{\prime},0;\theta )}(t)){\eta^{i}}_{(z^{\prime},0;\theta )}(t) {\eta^{j}}_{(z^{\prime},0;\theta )}(t)dt dz^{\prime} =0$$
	 
	\noindent  By following their arguments verbatim, we get the following equation :
	\begin{equation}\label{eq:1}
		\int e^{i\lambda z^\prime(x,\theta^\prime).\xi^\prime}a_{N}(x,\theta^\prime)f_{ij}(x){b^{i}}(x,\theta^\prime) {b^{j}}(x,\theta^\prime)dx =0
	\end{equation}
	
	\noindent Here, $(x,\theta^\prime)\to a_{N}$ is analytic and satisfies the estimate: 
	\begin{align}	\label{eq:estimate on cutoff}
	|\partial^{\alpha}a_{N}|\leq(CN)^{|\alpha|}, \quad \alpha \leq N,
	\end{align} 
	 Also, note that $b(0,\theta^\prime)={\eta}\in \theta^{\perp}$ and $a_{N}(0,\theta^\prime)=1$.
	
	\noindent To fix ideas let $\xi_0=e_{n-1}$ and the $\theta_0$ corresponding to the geodesic $\gamma_0$ be $e_n$ as in the proof of \cite[Proposition 2]{SU3}. Further, let us choose $\theta(\xi)$ to be a vector depending analytically on $\xi$ in a neighbourhood of the geodesic $\gamma_0$ in the following way: $$\theta(\xi)=\left(\xi_1,\dots,\xi_{n-2},-\frac{\xi_{1}^{2}+\dots+\xi_{n-2}^{2}+\xi_{n}}{\xi_{n-1}},1\right).$$ Since, $\theta $ can be made to analytically depend on $\xi$, we can also make the hyperplane $\theta^{\perp}$ depend analytically on $\xi$. This is to say that every $\eta$ which is orthogonal to $\gamma$ and is parallel along it  also depends analytically on $\xi$. Indeed, this is to be expected as we showed in the proof of Lemma \ref{eta_analytic}, that $\eta(t)$ depends analytically on the initial conditions, in particular on $\theta$ and since $\theta$ depends on $\xi$ analytically, so will $\eta$ depend on $\xi$ analytically. In fact, this can be achieved by carrying out a Gram Schmidt orthogonalization procedure. Consider the vectors: $v_1,\dots,v_{n-1},\theta(\xi)$ such that these form a basis at $T_{\gamma(0)}M$. Next, let $\eta_1(0;\xi)=v_1-\operatorname{proj}_{\theta(\xi)}v_1$. This depends analytically on $\xi$ because $\theta(\xi)$ does and because the projection map is also analytic. Now $\eta_2(0;\xi)=v_2-\operatorname{proj}_{\eta_1(0;\xi)}v_2-\operatorname{proj}_{\theta(\xi)}v_2$. We can construct $\eta_3(0;\xi),\dots,\eta_{n-1}(0;\xi)$ in a similar fashion, and all $\eta_i(0;\xi)$ will depend on $\xi$ analytically. Now the parallel translate of $\eta$ along $\gamma$ can also be shown to depend on $\xi$ analytically through its analytical dependence on initial conditions. It is also clear that if $\eta_i(t) \text{ and } \eta_j(t)$ depend analytically on $\xi$, then any linear combination also depends analytically on $\xi$. \\ \noindent Now with the choice of $\theta(\xi)$ made as above near $\xi=\xi_{0}$, it is easy to check that the following conditions are satisfied :
	\begin{align*}
	\theta(\xi)\cdot \xi&=0, \quad {\theta}^{n}(\xi)=1\quad \quad \text{ and }\\
	\theta(\xi_{0})&=(0,\dots,1)=e_{n}
	\end{align*}
    We will rewrite \eqref{eq:1} using the above mapping in the following form:
	\begin{align}\label{eq:2}
	\int e^{i\lambda \phi(x,\xi)}\tilde{a_{N}}(x,\xi)f_{ij}(x){\tilde{b}}^{i}(x,\xi){\tilde{b}}^{j}(x,\xi)dx =0.
	\end{align}
	\noindent Here $\phi(x,\xi)=z^\prime\cdot\xi^\prime$ is the phase function,  $\tilde{b}(x,\xi)= \eta_{(z^{\prime},0;\theta(\xi) )}(t)$, $t=t(x,\theta(\xi))$ and $z^{\prime}=z^{\prime}(x,\theta(\xi))$, see \cite{SU3}. This phase function has been shown in \cite{SU3} to be non-degenerate in a neighborhood of the geodesic $\gamma_0$.
	
	\noindent Now we quote the following lemma from \cite{SU3}:
	\begin{Lemma} \cite[Lemma 3.2]{SU3}\label{Lemma:1}
		Let, ${\theta(\xi)}$ and $\phi(x,\xi)$ be as above. Then, $\exists$ $ \delta>0$ such that if $$\phi_{\xi}(x,\xi)=\phi_{\xi}(y,\xi)$$ for some $x \in U$, $|y|<\delta$, $|\xi-\xi_0|<\delta$ where $\xi$ is complex, then $y=x$.
	\end{Lemma}
	\noindent We will study the analytic wavefront set of $f$ using Sj\"ostrand's complex stationary phase method. For this assume $x$, $y$ as in Lemma \ref{Lemma:1} and $|\xi_0-\upsilon|<\frac{\delta}{\tilde{C}}$ with $\tilde{C}>>2$ and $\delta<<1$.  Multiply \eqref{eq:2} by
	$$\tilde{\chi}(\xi-\upsilon)e^{i\lambda\left(i\frac{(\xi-\upsilon)^2}{2}-\phi(y,\xi)\right)}$$
	where $\tilde{\chi}$ is the characteristic function of the ball $B(0,\delta)\subset \mathbb{C}^n$ and then integrate w.r.t. $\xi$ to get : 
	\begin{align}\label{eq:3}
	\iint e^{i\lambda \Phi(y,x,\xi,\upsilon)}\tilde{\tilde{a_{N}}}(x,\xi)f_{ij}(x){\tilde{b}}^{i}(x,\xi){\tilde{b}}^{j}(x,\xi)dxd\xi =0.
	\end{align}
	In the above equation, $\tilde{\tilde{a_{N}}}=\tilde{\chi}(\xi-\upsilon)\tilde{a}_{N}$ is another analytic and elliptic amplitude for $x$ close to zero and $|\xi-\upsilon|<\frac{\delta}{\tilde{C}}$ and $$\Phi=-\phi(y,\xi)+\phi(x,\xi)+\frac{i}{2}(\xi-\upsilon)^{2}.$$ Furthermore, $$\Phi_{\xi}=\phi_{\xi}(x,\xi)-\phi_{\xi}(y,\xi)+i(\xi-\upsilon).$$
	To apply the stationary phase method we need to know the critical points of $\xi \mapsto \Phi$. Using lemma \ref{Lemma:1} above we have :
	\begin{enumerate}
		\item If $y=x$, $\exists$ a unique real critical point $\xi_{c}=\upsilon$
		\item If $y \neq x$, there are no real critical points 
		\item Also by Lemma \ref{Lemma:1}, if $y \neq x$, there is a unique complex critical point if $|x-y|<\delta/C_{1}$ and no critical points for $|x-y|>\delta/C_0$ for some constants $C_0 \text{ and }C_1 $ with $C_1 > C_0$.		
	\end{enumerate}
	Define, $\psi(x,y,\upsilon):=\Phi(\xi_c)$. Then at $x=y$
	\begin{center}
		(i) $\psi_{y}(x,x,\upsilon) = -\phi_{x}(x,\upsilon)$ \quad
		(ii) $\psi_{x}(x,x,\upsilon) = \phi_{x}(x,\upsilon)$ \quad
		(iii) $\psi(x,x,\upsilon) =0 $. \quad
	\end{center}
	Now, we split the $x$ integral in \eqref{eq:3} in to two parts : we integrate over $\{x:|x-y|>\delta/C_{0}\}$ for some $C_{0}>1$ and its complement. Since, $|\Phi_{\xi}|$ has a positive lower bound for $\{x:|x-y|>\delta/C_{0}\}$ and there are no critical points of $\xi \to  \Phi$ in this set, we can estimate that integral in the following manner: First note that, $e^{i\lambda \Phi(x,\xi)} = \frac{\Phi_{\xi}\partial_{\xi}}{i\lambda|\Phi_\xi|^{2}}e^{i\lambda \Phi(x,\xi)}$. Using, \eqref{eq:estimate on cutoff} and integrating by parts $N$ times with respect to $\xi$ and the fact that on the boundary $|\xi-\upsilon|=\delta$, we get
	\begin{align}\label{eq:4}
	\left|	\iint_{|x-y|>\delta/C_{0}} e^{i\lambda \Phi(y,x,\xi,\upsilon)}\tilde{\tilde{a_{N}}}(x,\xi)f_{ij}(x){\tilde{b}}^{i}(x,\xi){\tilde{b}}^{j}(x,\xi)dxd\xi \right|\leq C\bigg(\frac{CN}{\lambda}\bigg)^{N}+CNe^{-\frac{\lambda}{C}}
	\end{align}
	
	\noindent We choose $N\leq\lambda/Ce\leq N+1$ to get an exponential error on the right. Now in estimating the integral
	\begin{align}\label{eq:5}
	\left|	\int_{|x-y|\leq\delta/C_{0}} e^{i\lambda \Phi(y,x,\xi,\upsilon)}\tilde{\tilde{a_{N}}}(x,\xi)f_{ij}(x){\tilde{b}}^{i}(x,\xi){\tilde{b}}^{j}(x,\xi)dxd\xi \right|
	\end{align}
	we use \cite[Theorem 2.8]{Sjostrand} and the remark following that to conclude: 
	\begin{align}\label{eq:6}
	\int_{|x-y|\leq\delta/C_{0}}e^{i\lambda \psi(x,\alpha)}f_{ij}(x)B^{ij}(x,\alpha;\lambda)dx=\mathcal{O}(e^{-\lambda/C})
	\end{align}
	where  $\alpha=(y,\upsilon)$ and $B$ is a classical analytical symbol with principal part $\tilde{b}\otimes\tilde{b}$.  See appendix of \cite{Abhishek2017} for a proof of estimates in (\ref{eq:4}) and (\ref{eq:6}).
	\par \noindent Let, $\beta=(y,\mu)$ where, $\mu=\phi_{y}(y,\upsilon)=\upsilon+\mathcal{O}(\delta)$. At $y=0$, we have $\mu=\upsilon$. Also $\alpha \to \beta$ is a diffeomorphism following similar analysis as in \cite[Section 4]{SU3}. If we write $\alpha=\alpha(\beta)$, then the above equation becomes: 
	\begin{align}\label{eq:7}
	\int_{|x-y|\leq\delta/C_{0}}e^{i\lambda \psi(x,\beta)}f_{ij}(x)B^{ij}(x,\beta;\lambda)dx=\mathcal{O}(e^{-\lambda/C})
	\end{align}
	where $\psi$ satisfies (i), (ii) and (iii), and $B$ is a classical analytical symbol as before and :
	\begin{center}
		$\psi_{y}(x,x,\upsilon) = -\mu$,\quad 
		$\psi_{x}(x,x,\upsilon) = \mu$ \quad and \quad 
		$\psi_{y}(x,x,\upsilon) =0 $
	\end{center}
	The symbols in \eqref{eq:7} satisfy : $$\sigma_{P}(B)(0,0,\mu)= {\eta}(\mu) \otimes{\eta}(\mu) = \eta^{\otimes 2}(\mu)$$ where $\eta$ is a vector perpendicular to $\theta$ as before.
	
	\par \noindent We will now show that there exists $N=\frac{n(n+1)}{2}$  unit vectors at $x_0$, say, ${\eta}_1, {\eta}_2,\dots,{\eta}_N$, such that for each $\eta_p $ there exists  some vector perpendicular to $\eta_p$ in a small neighbourhood of $\theta_0=e_n$ such that the symmetric $2$- tensor $f_{ij}$ is uniquely determined by the numbers $f_{ij}(\eta_p)_i(\eta_p)_j$. First of all, let  $$S=\{\eta: \exists \theta \text{ in nbd.}(\theta_0)\text{ such that }\theta\in \xi_0^{\perp}, \langle\theta,\eta\rangle=0 \}$$
	We will first show that for $n\geq 3$, $\operatorname{Span}(S)=\mathbb{R}^{n}$. Let $\theta_0 = e_n$ and $\xi_0= e_{n-1}$ as has been the case in the proof of this proposition. Then it is immediately obvious that $e_i=\eta_i \in S$ for $i=1,\dots,n-1$. Now consider a vector $v=\epsilon e_1+e_n \in \xi_0^{\perp}$ such that $v \in nbd.(\theta_0)$. Then $\eta_n=e_1-\epsilon e_n$ also belongs to $S$. Clearly now $\operatorname{Span}(S)=\mathbb{R}^n$. Hence, $\exists \eta_1,\dots,\eta_n$ in the set $S$ which are linearly independent. Now we make the following claim:
	\begin{claim}
		There exists a linear combination $a_p\eta_p+a_q\eta_q$, $p,q= 1,\dots,n$ and $p<q$ for each $p$ and $q$ such that $a_p\eta_p+a_q\eta_q \in S$. Hence there exists $\binom{n}{2}$ such linear combinations for different pairs of $p$ and $q$ which can be listed as $\eta_l$,  $l=(n+1),\dots,\frac{n(n+1)}{2}$.
	\end{claim}
\begin{proof}
	Our goal is to show that for each pair $\eta_p, \eta_q$ as stated in the claim, $\exists \theta \text{ in nbd.}(\theta_0)$ such that $\theta\in \xi_0^{\perp}, \langle\theta,(a_p\eta_p+a_q\eta_q)\rangle=0 $. Since, $\eta_p$,$\eta_q$ are in $S$, hence there must exist $\theta_p$ and $\theta_q$ in $nbd.(\theta_0)\cap \xi_0^{\perp}$such that $\langle\theta_p,\eta_p\rangle=0 $ and $\langle\theta_q,\eta_q\rangle=0 $. Now we have 3 cases to consider:\\
	\textbf{Case 1}: $\langle\theta_q,\eta_p\rangle=0 $ and $\langle\theta_p,\eta_q\rangle=0$.\\
	In this case, it is clear that any linear combination $\eta_p$ and $\eta_q$ is in $S$. This is because $\langle \theta_p, (\eta_p+\eta_q)\rangle=\langle \theta_q (\eta_p+\eta_q)\rangle=0$.\\
	\textbf{Case 2}: Only one of $\langle\theta_q,\eta_p\rangle $ or $\langle\theta_p,\eta_q\rangle$ is zero.\\ Without the loss of generality, let $\langle\theta_q,\eta_p\rangle =0 $. In this case also, any linear combination $a_p\eta_p+a_q\eta_q \in S$. This is because, $\langle \theta_q,(a_p\eta_p+a_q\eta_q)\rangle=0$.\\
	\textbf{Case 3}: $\langle\theta_q,\eta_p\rangle \neq 0 $ and $\langle\theta_p,\eta_q\rangle \neq 0$\\
	Consider the vector $\theta_p+\epsilon \theta_q$ where $\epsilon\neq0$ is chosen such that $\theta_p+\epsilon \theta_q \in nbd.(\theta_0)\cap \xi_0^{\perp}$. Now let $a_p\eta_p+a_q\eta_q$ be such that $a_p=1$ and $a_q=-\frac{\langle \theta_p,\eta_q\rangle}{\epsilon(\langle \theta_q,\eta_p \rangle)}$. Then it can be readily seen that $\langle (\theta_p+\epsilon \theta_q),(a_p\eta_p+a_q\eta_q)=0$ for the above mentioned choice for $a_p$ and $a_q$. Hence this $a_p\eta_p+a_q\eta_q \in S$. \\
	This proves our claim.
\end{proof}
\noindent So consider the collection of $N=\frac{n(n+1)}{2}$ vectors $\eta_k$ as listed above for $ k=1,\dots,N$ (where we represent $\frac{\eta_k}{|\eta_k|}$ as $\eta_k$ by a slight abuse of notation). The set $\{\eta_k^{\otimes2}\}_{k=1}^{k=N}$ is linearly independent and determines $f_{ij}$ from the numbers: $f_{ij}\eta_k^i\eta_k^j$ for $k=1,\dots,N$.
 Coming back to the proof of Proposition \ref{SUanalogue}, we can get $N$ equations of the kind \ref{eq:7}, with symbols $B_k$ such that $\sigma_p(B_k)=(\eta_k)^{\otimes 2}$ for $k=1,\dots,N$. Together these $N$ equations can be thought of as a tensor valued operator applied to tensor $f$, and then by similar analysis as in the proof of  \cite [Proposition 2]{SU3}, we conclude the proof of Proposition \ref{SUanalogue}.
\end{proof}

\section {Proof of the support theorem}

\begin{proof}[Proof of Theorem \ref{Main_theorem}]
	We will use similar ideas as in \cite{K1} for the proof. We extend $f$ outside $M$ by zero, i.e. $f=0$ in $\widetilde{M}\setminus M$. Let $\gamma_0$ be a fixed geodesic in $\mathcal{A}$ which is continuously deformable within the set $\mathcal{A}$ to some geodesic $\gamma_1$ in $\mathcal{A}$ that does not intersect $\text{supp}(f)$.  Let $\gamma_t$ be an intermediate geodesic in the deformation. As $f$ is compactly supported within $M$ and $\mathcal{A}$ is open, hence such a $\gamma_1$ can always be found. Furthermore, we parametrize these geodesics by their starting points $x\in \partial{\widetilde{M}}$ and their initial directions $\theta \in S^{n-1}$. Accordingly, let $(x_t,\theta_t)$ denote the starting point and the initial direction for the geodesic $\gamma_t$. Consider a ``cone of geodesics" formed by geodesics around $\gamma_t$ having the same starting point as $\gamma_t$ i.e. $x_t$ and with initial directions that are sufficiently close to $\theta_t$ such that this cone is still in $\mathcal{A}$. Clearly, by a compactness argument, there is one such cone that does not intersect $\text{supp}(f)$. Now carry out the construction of such cones for each $\gamma_t$ in the aforementioned deformation, i.e. for all values of $t\in [0,1]$. We will call these cones $C_t$. Let $$t_1= \inf\{t\in[0,1]:C_{t_2}\cap \text{supp}(f)=\emptyset, \quad \forall t_2>t\}.$$
	Suppose that $t_1>0$. Then the cone $C_{t_1}$ intersects $\text{supp}(f)$ at some point, say $p_0$. By the Sato- Kawai- Kashiwara Theorem, $(p_0,\zeta_0) \in WF_A(f)$ where $\zeta_0$ is normal to the cone $C_{t_1}$. But this counters Proposition \ref{SUanalogue}. This shows that $f$ is zero in a neighborhood of $p_0$. By a compactness argument this shows that $f=0$ for a positive distance away from the cone $C_{t_1}$, see \cite{K1}. But this would contradict the claim that $t_1>0$ is the infimum. Hence $t_1=0$ which completes the proof of our main theorem.
\end{proof}
\bibliographystyle{plain}
\bibliography{article_trt_april10}

\def\dbar{\leavevmode\hbox to 0pt{\hskip.2ex \accent"16\hss}d}
\begin{thebibliography}{10}

\bibitem{Abhishek2017}
Anuj Abhishek and Rohit~Kumar Mishra.
\newblock Support theorems and an injectivity result for integral moments of a
  symmetric m-tensor field.
\newblock (Submitted).

\bibitem{BQ}
Jan Boman and Eric~Todd Quinto.
\newblock Support theorems for real-analytic {R}adon transforms.
\newblock {\em Duke Math. J.}, 55(4):943--948, 1987.

\bibitem{Desai2016}
Naeem~M. Desai and William R.~B. Lionheart.
\newblock An explicit reconstruction algorithm for the transverse ray transform
  of a second rank tensor field from three axis data.
\newblock {\em Inverse Problems}, 32(11):115009, 19, 2016.

\bibitem{Desai2016a}
Naeem~M. Desai and William R.~B. Lionheart.
\newblock An explicit reconstruction algorithm for the transverse ray transform
  of a second rank tensor field from three axis data.
\newblock {\em Inverse Problems}, 32(11):115009, 19, 2016.

\bibitem{Hammer2004}
H.~Hammer and B.~Lionheart.
\newblock Application of sharafutdinov's ray transform in integrated
  photoelasticity.
\newblock {\em Journal of Elasticity}, 75(3):229--246, Jun 2004.

\bibitem{K1}
Venkateswaran~P. Krishnan.
\newblock A support theorem for the geodesic ray transform on functions.
\newblock {\em J. Fourier Anal. Appl.}, 15(4):515--520, 2009.

\bibitem{KS}
Venkateswaran~P. Krishnan and Plamen Stefanov.
\newblock A support theorem for the geodesic ray transform of symmetric tensor
  fields.
\newblock {\em Inverse Problems and Imaging}, 3(3):453--464, 2009.

\bibitem{Lee2003}
John~M Lee.
\newblock Smooth manifolds.
\newblock In {\em Introduction to Smooth Manifolds}, pages 1--29. Springer,
  2003.

\bibitem{Lee2006}
John~M Lee.
\newblock {\em Riemannian manifolds: an introduction to curvature}, volume 176.
\newblock Springer Science \& Business Media, 2006.

\bibitem{Lionheart2015a}
W.~R.~B. Lionheart and P.~J. Withers.
\newblock Diffraction tomography of strain.
\newblock {\em Inverse Problems}, 31(4):045005, 17, 2015.

\bibitem{Novikov2007}
Roman Novikov and Vladimir Sharafutdinov.
\newblock On the problem of polarization tomography. {I}.
\newblock {\em Inverse Problems}, 23(3):1229--1257, 2007.

\bibitem{Sharafutdinov}
V.~A. Sharafutdinov.
\newblock “ray transform on riemannian manifolds,” lecture notes.
\newblock http://www.math.nsc.ru/~sharafutdinov/files/Lectures.pdf.

\bibitem{Sharafutdinov_Book}
V.~A. Sharafutdinov.
\newblock {\em Integral geometry of tensor fields}.
\newblock Inverse and Ill-posed Problems Series. VSP, Utrecht, 1994.

\bibitem{Sharafutdinov2007}
Vladimir Sharafutdinov.
\newblock Slice-by-slice reconstruction algorithm for vector tomography with
  incomplete data.
\newblock {\em Inverse Problems}, 23(6):2603, 2007.

\bibitem{Sjostrand}
Johannes Sj{\"o}strand.
\newblock Singularit\'es analytiques microlocales.
\newblock In {\em Ast\'erisque, 95}, volume~95 of {\em Ast\'erisque}, pages
  1--166. Soc. Math. France, Paris, 1982.

\bibitem{SU2}
Plamen Stefanov and Gunther Uhlmann.
\newblock Boundary rigidity and stability for generic simple metrics.
\newblock {\em J. Amer. Math. Soc.}, 18(4):975--1003 (electronic), 2005.

\bibitem{SU3}
Plamen Stefanov and Gunther Uhlmann.
\newblock Integral geometry on tensor fields on a class of non-simple
  {R}iemannian manifolds.
\newblock {\em Amer. J. Math.}, 130(1):239--268, 2008.

\bibitem{Taylor2010}
Michael Taylor.
\newblock {\em Partial Differential Equations I}.
\newblock Springer-Verlag New York, 2010.

\end{thebibliography}

\end{document}